\documentclass[reqno]{amsproc}
\usepackage[utf8]{inputenc}
\usepackage{orcidlink}
\usepackage{amsfonts}
\usepackage{graphicx}
\usepackage{verbatim}
\usepackage{amscd}
\usepackage{amsmath, physics}
\usepackage{amssymb}
\usepackage{latexsym}
\usepackage{hyperref}
\usepackage[all]{xy}
\usepackage{color}
\usepackage{mathrsfs}
\newtheorem{theorem}{\bf Theorem}
\newtheorem{proposition}{\bf Proposition}

\newtheorem{lemma}{\bf Lemma}
\theoremstyle{definition}
\newtheorem{remark}{\bf Remark}
\newtheorem{definition}{\bf Definition}
\newtheorem{example}{\bf Example}

\usepackage{empheq,etoolbox}
\numberwithin{equation}{section}
\patchcmd{\subequations}
  {\theparentequation\alph{equation}}
  {\theparentequation.\alph{equation}}
  {}{}
  
\begin{document}
\title[A note on Huisken monotonicity-type formula for the MCF]{A note on Huisken monotonicity-type formula for the mean curvature flow in a gradient shrinking extended Ricci soliton background}

\author[José N.V. Gomes]{José N.V. Gomes$^1$\orcidlink{0000-0001-5678-4789}}
\author[Matheus Hudson]{Matheus Hudson$^2$\orcidlink{0000-0001-9786-1234}}
\author[Hikaru Yamamoto]{Hikaru Yamamoto$^3$\orcidlink{0000-0002-7047-2914}}

\address{$^{1}$Department of Mathematics, Universidade Federal de São Carlos, Rod. Washington Luís, Km 235, 13565-905 São Carlos, São Paulo, Brazil.}
\address{$^2$ Institut für Mathematik, Carl von Ossietzky Universität Oldenburg, 26129 Oldenburg, Germany.}
\address{$^3$Department of Mathematics, Faculty of Pure and Applied Science, University of Tsukuba, 1-1-1 Tennodai, Tsukuba, Ibaraki 305-8571, Japan.}

\email{$^1$jnvgomes@ufscar.br}
\email{$^{2}$matheus.hudson.gama.dos.santos@uol.de}
\email{$^3$hyamamoto@math.tsukuba.ac.jp}

\urladdr{$^{1}$https://www.ufscar.br}
\urladdr{$^2$https://uol.de}
\urladdr{$^3$https://nc.math.tsukuba.ac.jp}


\begin{abstract}
We give an application of a Huisken monotonicity-type formula for the mean curvature flow in a compact smooth manifold with a Riemannian metric that evolves by a shrinking self-similar solution of the extended Ricci flow. Our investigation builds on previous articles by Huisken and the third author, as we apply their techniques to establish new results in this geometric setting. Moreover, under some natural geometric assumptions, the noncompact case is also solved.
\end{abstract}
\maketitle

\section{Introduction}
An important geometric flow in the setting of Riemannian manifolds is the celebrated mean curvature flow, which falls in the class of extrinsic geometric flows. It is well known that the mean curvature flow (MCF, for short) has been a constant object of investigation, and it has experienced a great development in the last few decades. Here, it is worth mentioning some results in the literature that are related to our work. Huisken~\cite{huisken1990asymptotic} showed that the shrinking self-similar solutions to the MCF in Euclidean space appear as stationary points for the Gaussian area-type functional playing the role of the energy-type functional, which is non-increasing along the flow. Lott~\cite{lott2012mean} worked on the MCF in a compact smooth manifold with boundary and Riemannian metric that evolves by Ricci flow, and then he found a Huisken monotonicity-type formula in the special case of self-similar solutions to the Ricci flow. Magni, Mantegazza and Tsatis~\cite{magni2013flow} also found a Huisken monotonicity-type formula for the MCF in an ambient smooth manifold with Riemannian metric that evolves by a self-similar solution to the Ricci flow. Recently, the first and second authors~\cite{gomes2023mean} showed a Huisken monotonicity-type formula for the MCF in an ambient space with a Riemannian metric that evolves by a self-similar solution of the extended Ricci flow.

In \cite{magni2013flow}, special emphasis was given for the possible generalization of Huisken's monotonicity formula and its connection with the validity of some Li-Yau-Hamilton differential Harnack-type inequalities in the setting of submanifolds in an ambient Riemannian manifold evolving by Ricci or backward Ricci flow, whereas in \cite{lott2012mean} it was proved a relation between the mean curvature solitons and an extension of Hamilton’s differential Harnack expression which vanishes on the steady case, recovering classical results known for translating solitons to the MCF in Euclidean space (see Hamilton~\cite[Def.~4.1 and Lem.~3.2]{hamilton1995harnack}). In~\cite{gomes2023mean}, special attention has been given for the mean curvature solitons, for instance, an extension of Hamilton's differential Harnack expression naturally appears in the more general context of mean curvature solitons in a compact smooth manifold with boundary and Riemannian metric evolving by extended Ricci flow, whose characterization of nullity should be on the steady case. Besides, it was given a way on how to construct examples of these mean curvature solitons, and also, a characterization of an arbitrary family of such solitons was proved. Before that, the third author showed a characterization of a family of self-shrinkers in a gradient shrinking Ricci soliton (see~\cite[Sect.~4]{yamamoto2020meancurvature}).

Huisken also proved, by means of his monotonicity formula, that the MCF converges, after rescaling, to a self-shrinker in Euclidean space when it develops a singularity of type-I, i.e., the growth rate estimate for the norm of the second fundamental form is bounded. In the same way, the third author gave an application of the Huisken monotonicity-type formula obtained by Lott in the case of gradient shrinking Ricci solitons. He showed that the MCF converges to a self-shrinker soliton with the additional uniformity conditions of bounded geometry of the ambient space.

In this note, we follow the approach in \cite{yamamoto2020meancurvature} to give an application of the Huisken monotonicity-type formula for the context of the MCF in a gradient shrinking extended Ricci soliton background obtained in \cite{gomes2023mean}. The most appropriate setting to do this is to consider the family $\mathscr F$ and its associated normalized family $\widetilde{\mathscr F}$ as defined in Remark~\ref{Setting-MCF}. Precisely, the main upshot of our investigation is the following convergence theorem.

\begin{theorem}\label{application:Huisken}
Assume that $(M,g)$ is an $n(\geqslant3)-$dimensional compact Riemannian manifold, and let $(\overline{g}(t),\overline{w}(t))$ be a shrinking self-similar solution to the extended Ricci flow on $M$ with potential function $\overline{f}$ and initial value $(g,w)$. Given an $(n-1)-$dimensional compact smooth manifold $\Sigma$ without boundary, and let $\mathscr F$ be the MCF of $\Sigma$ in a gradient shrinking extended Ricci soliton background which develops a singularity of type-I. Consider the  normalized MCF $\widetilde{\mathscr{F}}$ in $(M,g)$. Then, for any sequence $s_{1}<s_{2}<\cdots<s_{j}<\cdots \rightarrow \infty$ and points $\left\{p_{j}\right\}_{j=1}^{\infty}$ in $\Sigma$, there exist subsequences $s_{j_{k}}$ and $p_{j_{k}}$ such that the family of immersion maps $\widetilde x_{s_{j_k}}: \Sigma \to (M,g)$ from pointed manifolds $\big(\Sigma,  p_{j_{k}}\big)$ converges to an immersion map $x_\infty: \Sigma_\infty \to (M,g)$ from an $(n-1)-$dimensional complete pointed Riemannian manifold $\left(\Sigma_{\infty}, x_\infty^*g, p_{\infty}\right)$ in the $C^\infty$ Cheeger-Gromov sense. Furthermore, $(\Sigma_{\infty}, x_\infty^* g)$ is an $f_\infty-$minimal hypersurface of  $(M, g),$ where $f_\infty = f \circ x_\infty$. 
\end{theorem}

This note is structured as follows. We begin in Section~\ref{SecMCFERFB} by defining and commenting upon the concepts required to lay the groundwork for our proofs. In Section~\ref{sect:apllication}, we proceed with the proofs of preliminary results which are formulated in the more general context of complete Riemannian manifolds with bounded geometry. This section ends with the proof of the main theorem. In Section~\ref{Sec-ncompcase} we prove our main theorem for the case of a complete noncompact Riemannian manifold with some additional uniformity conditions. We conclude this note by appending an example of an $f$-minimal hypersurface of a Euclidean spherical cap and showing how to construct a family of mean curvature solitons for the MCF in a gradient extended Ricci soliton background by means of radial smooth functions on the Euclidean space.

\section{Mean curvature flow in an extended Ricci flow background}\label{SecMCFERFB}

Consider an $n (\geqslant 3)-$dimensional smooth manifold $M$ and a solution $(g(t), w(t))$ to the \emph{extended Ricci flow}
\begin{align}\label{ListEq}
\left\{
\begin{array}{lcl}
\frac{\partial}{\partial t}g(t) = - 2\operatorname{Ric}_{g(t)} + 2\alpha_n dw(t)\otimes dw(t),\\[1ex]
\frac{\partial}{\partial  t}w(t) =  \Delta_{g(t)} w(t),
\end{array}
\right.
\end{align}
in $M \times I,$ for some initial value $(g,w).$ Here and throughout the paper, $\alpha_n=(n - 1)/(n - 2)$, $\operatorname{Ric}_{g(t)}$ stands for the Ricci tensor of the Riemannian metric $g(t)$, the Laplacian operator $\Delta_{g(t)}$ is the trace of the Hessian operator $\nabla_{g(t)}^2$ computed on $g(t)$, and $dw(t)\otimes dw(t)$ denotes the tensor product of the $1$-form $dw(t)$ by itself, which is metrically dual to gradient vector field $\nabla w(t)$ computed on $g(t)$ of a scalar smooth function $w(t)$ on $M.$ 

It is worth noting that List established a close connection between the Ricci flow and the static Einstein vacuum equations via the extended Ricci flow, thereby justifying the choice of $\alpha_n$ (see \cite[pp.~1010-1012]{Bernhard_List} for details). This provides an interesting and useful link between problems in low-dimensional topology and geometry and questions arising in general relativity. For a detailed account of the extended Ricci flow, including proofs of short- and long-time existence of solutions to \eqref{ListEq}, we refer to~\cite[Sects.~4 and 5]{Bernhard_List}.

Throughout the paper, $\alpha_n$ may be taken to be any positive constant. However, we prefer to keep the original version of the extended Ricci flow, as it is already well established in List’s framework. We shall follow the notation and terminology of \cite{gomes2023mean} and~\cite{yamamoto2020meancurvature}. In particular, whenever there is no danger of confusion, we simplify the notation by suppressing the parameter $t.$

A gradient extended Ricci soliton on $M$ is a self-similar solution $(\overline{g}(t), \overline{w}(t) )$ of~\eqref{ListEq} given by
\begin{align*}
\left\{
\begin{array}{lcl}
\overline{g}(t)  = \sigma(t)\psi_t^* g,\\ [1ex]
\overline{w}(t) = \psi_t^* w,
\end{array}
\right.
\end{align*}
for some initial value $(g,w)$, where $\psi_t$ is a  smooth one-parameter family of diffeomorphisms of $M$ generated from the flow of $\nabla_g f/\sigma(t)$, for some $f\in C^\infty(M)$, and  $\sigma(t)$ is a smooth positive function on $t$. By setting $\overline{f}(t)=\psi_t^*f$,  from~\eqref{ListEq} we have
\begin{align}\label{mod_grad_Ricci_soliton}
\left\{
\begin{array}{lcl}
\operatorname{Ric}_{\overline{g}} + \nabla_{\overline{g}}^2\,\overline{f} - \alpha_n  d\overline{w} \otimes d\overline{w} = \dfrac{c}{2(T - t)}\overline{g},\\
\Delta_{\overline{g}} \overline{w} = \langle\nabla_{\overline{g}} \overline{f}, \nabla_{\overline{g}} \overline{w}\rangle_{\overline{g}},
\end{array}
\right.
\end{align}
where $c= 0$ in the steady case, for $t \in  \mathbb{R};$ $c = 1$ in the shrinking case, for $t \in (-\infty, T)$;  and $c = 1$ in the expanding case, for $t \in (T,+\infty).$ Moreover,
\begin{align}\label{self-solution}
\dfrac{\partial }{\partial t} \overline{f} = \left|\nabla_{\overline{g}} \overline{f}\right|_{\overline{g}}^2\,.
\end{align}
The function $\overline{f}$ is the so-called \emph{potential function}. 

As in \cite{gomes2023mean}, we consider the MCF in an ambient space with a complete Riemannian metric that evolves by an extended Ricci flow, as follows: let $(g(t),w(t))$ be an extended Ricci flow in $M \times I$. Given an $(n-1)-$dimensional compact smooth manifold $\Sigma$ without boundary, let $\{x(\cdot,t); \; t\in [0,T)\}$ be a smooth one-parameter family of immersions of $\Sigma$ in $M$.  For each $t\in [0,T),$ set  $x_t = x(\cdot, t)$ and $\Sigma_t$ for the hypersurface $x_t(\Sigma)$ of $(M,g(t)),$ which can be written as $\Sigma_t:=(\Sigma, x^*_t g(t))$. We say that the family $\mathscr F :=\{\Sigma_t;\; t\in [0,T)\}$ is a \emph{MCF in an extended Ricci flow background} if it evolves under MCF
\begin{align*}
\left\{
\begin{array}{rcl}
\frac{\partial}{\partial t}x(p,t) &=&  H(p,t) e(p,t),\\ [1ex]
x(p, 0) &=& x_0(p),
\end{array}
\right.
\end{align*}
where $H(p,t)$ and $e(p,t)$ are the mean curvature and the unit normal of $\Sigma_t$ at the point $p\in\Sigma$, respectively. In the particular case where $(g(t),w(t))=(\overline{g}(t),\overline{w}(t))$ is a gradient extended Ricci soliton on $M$, with potential function $\overline{f}$, we say that $\mathscr F$ is a \emph{MCF in a gradient extended Ricci soliton background}, and a hypersurface $\Sigma_t\in \mathscr F$ is a \emph{mean curvature soliton}, if
\begin{align*}
H(p,t)+e(p,t)\overline{f}=0 \quad \forall p\in\Sigma.
\end{align*}
Here, $e(\,\cdot\,,t)$ must be the inward unit normal vector field on $\Sigma_t.$ Note that the interval $I$ has an appropriate choice in each case of the flow. Besides, whenever there is no danger of confusion, we are writing  $g(t)$ on $\Sigma$ instead of $x_t^*g(t).$

For constructing a family of mean curvature solitons for the MCF in a gradient extended Ricci soliton background on $M$ from a smooth one-parameter family of diffeomorphisms $\psi_t$ of $M$ generated from the flow of $\nabla_g f/\sigma(t)$ and initial value $(g, w)$, it is enough to consider an $f$-minimal hypersurface (i.e., $H_g+e_0(f)=0$) of $(M,g)$, and then to proceed as in \cite[Thm.~3]{gomes2023mean} to obtain such a family. In Section~\ref{Concluding remarks}, we give an example of an $f$-minimal hypersurface of an Euclidean spherical cap and show how to construct a family of mean curvature solitons for the MCF in a gradient extended Ricci soliton background by means of radial smooth functions on Euclidean space. For explicit examples of MCF in a gradient Ricci soliton background, see~\cite{yamamoto2018examples}.

The MCF in a gradient shrinking extended Ricci soliton background develops a \emph{singularity of type-I} when there exists a constant $C > 1$ such that 
\begin{align}\label{singularity:type:I}
\max_{p \in \Sigma} |\mathcal A(p, t)| \leqslant  \frac{C}{\sqrt{T - t}}, 
\end{align}
where $\mathcal A(\cdot, t)$ stands for the second fundamental form of $\Sigma_t.$

To provide an application of a Huisken monotonicity-type formula for the context of the MCF in a gradient shrinking extended Ricci soliton background, we rewrite Thm.~2 of \cite{gomes2023mean} in a convenient way, as follows.

\begin{theorem}\label{Huisken_monotonicity}
Assume that $(M,g)$ is an $n(\geqslant 3)-$dimensional complete Riemannian manifold, and let $(\overline{g}(t),\overline{w}(t))$ be a shrinking self-similar solution to the extended Ricci flow on $M$ with potential function $\overline{f}$ and initial value $(g,w)$. Given an $(n-1)-$dimensional compact smooth manifold $\Sigma$ without boundary, and let $\mathscr F$ be the MCF of $\Sigma$ in a gradient shrinking extended Ricci soliton background. Denote by $dA_{\overline{g}}$ the $(n-1)$-dimensional Riemannian measure on $\Sigma$ induced by $\overline g(t)$, and consider the function $\mathscr A(t)$ given by
\begin{equation*}
(-\infty,T)\ni t\mapsto [4\pi(T-t)]^{-( n - 1 ) / 2}\int_\Sigma e^{-\overline{f}}dA_{\overline{g}}.   
\end{equation*}
Then
\begin{equation*}
\frac{d}{dt}\mathscr A(t) 
=  -  [4\pi(T-t)]^{-( n - 1 ) / 2}\int_\Sigma \left( H_{\overline g} + e_t  \overline f \right)^2  e^{-\overline f}dA_{\overline g}
\end{equation*}
for all $t\in[0,T).$ In particular, $\mathscr A(t)$ is non-increasing. Moreover, $\mathscr A (t)$ is constant if and only if $\mathscr F$ is a family of mean curvature solitons. 
\end{theorem}

The following remark is worth noting to justify the main setting of this note.

\begin{remark}\label{Setting-MCF}
Assume that $(M,g)$ is an $n(\geqslant3)-$dimensional complete Riemannian manifold, and let $(\overline{g}(t),\overline{w}(t))$ be a shrinking self-similar solution to the extended Ricci flow on $M$ with potential function $\overline{f}$ and initial value $(g,w)$. Given an $(n-1)-$dimensional compact smooth manifold $\Sigma$ without boundary, and let $\mathscr F :=\{\Sigma_t; \; t\in [0,T)\}$ be the MCF of $\Sigma$ in a gradient shrinking extended Ricci soliton background. Write $H_{\overline g}=H(p,t),$  $e_t=e(p,t)$ and $\widetilde{x}_{s} = \psi_t\circ x_t,$ for short. Setting $s=-\log (T- t)$ and $\widetilde{\Sigma}_s = (\Sigma,\widetilde{x}_{s}^{*}g),$ one has $s\in [-\log T,\infty)$, $\frac{ds}{dt}=\frac{1}{T- t}$ and $\frac{dt}{ds}=T- t.$ Since $\overline g(t)= (T- t)\psi^*_t g$, we have $H_{\overline{g}} =\frac{1}{\sqrt{T- t}} H_{\psi^*_t g}$ and $e_t = \frac{1}{\sqrt{T- t}}e_{\psi^*_t g}$. Then
\begin{align*}
\frac{\partial \widetilde x_s}{\partial s}=&  \left[ \frac{d \psi_t}{dt} \circ x_t +(\psi_t)_* \left(\frac{\partial x_t}{\partial t}\right)\right]\frac{dt}{ds} =  \frac{\nabla_{g}f\circ\psi_t}{T - t}  \circ x_t \cdot \frac{dt}{ds} + (T-  t)(\psi_t)_* \big(H_{\overline{g}}e_t\big) \\
=& \big(\nabla_g f + H_g\big)(\widetilde x_s) .
\end{align*}
We call the family $\widetilde{\mathscr F}:=\{\widetilde \Sigma_s;\; s\in [-\log T,\infty)\}$ the \emph{normalized MCF} in $(M, g).$
\end{remark}

As an application of Theorem~\ref{Huisken_monotonicity}, we prove our main result. By taking $w$ to be constant in \eqref{ListEq}, we recover Theorem~1.5 in \cite{yamamoto2020meancurvature} from Theorem~\ref{application:Huisken}. The proof is by means of the Arzelà-Ascoli theorem, which provides a sequence of isometric immersions on an exhaustion of $\Sigma_\infty$ converging to a limiting global solution which is $f_\infty$-minimal hypersurface of $(M, g)$. We prove global supremum estimates depending only on the initial bounds on the Riemann tensor $\operatorname{Rm}_g$ and Hessian operator $\nabla_g^2 w$, whereas the interior estimates depend on the full $C^\infty$ norm of $\overline g$ (see Section~\ref{sect:apllication}). This is possible since the estimates for the derivatives of $\overline g$ and $\overline w$ can be combined in the right way.

\section{An application of the Huisken monotonicity-type formula}\label{sect:apllication}

In this section, we prove Theorem~\ref{application:Huisken}, which is an application of a Huisken monotonicity-type formula for the context of MCF in a gradient extended Ricci soliton background. We begin with a brief discussion that establishes the groundwork for the compact and noncompact cases.

\begin{definition}[Bounded geometry]\label{DefBGeom}
We say that a complete Riemannian manifold $(M,g)$ has bounded geometry if for every integer $j\geqslant 0$ there exist positive constants $C_j$ and $\eta$ such that the Riemann curvature tensor $\operatorname{Rm}_g$ and the injectivity radius $\operatorname{inj}(M, g)$ satisfy
\begin{equation*}
|\nabla_g^j\operatorname{Rm}_g| < C_j  \quad \mbox{and} \quad \operatorname{inj}(M, g) \geqslant \eta>0.
\end{equation*}
\end{definition}

When $(g,w)$ is the initial value of a self-similar solution $(\overline{g}(t),\overline{w}(t))$ to the extended Ricci flow on a noncompact smooth manifold, we also assume that there exist positive constants $C'_j$ such that 
\begin{equation}\label{Additional-BG}
|\nabla_g^j (dw \otimes dw)| \leqslant C'_{j}
\end{equation}
for every integer $j \geqslant 0$. 

The motivation for this additional assumption stems from the proof of Theorem~7.5 by List (see item (b) of Lemma~7.6 in~\cite{Bernhard_List}). In particular, uniform bounds on all covariant derivatives of $w$ imply corresponding bounds on all covariant derivatives of $dw \otimes dw$. Of course, in the compact case, \eqref{Additional-BG} is satisfied.

Notice that $\mathbb{R}^n$ with the standard metric and compact Riemannian manifolds have bounded geometry. Actually, bounded geometry provides a natural generalization to the known settings of compact and Euclidean spaces. In particular, the class of manifolds of bounded geometry allows us to uniformly apply constructions that are well known for compact manifolds. For a way on how to construct complete manifolds of bounded geometry, we refer to Eldering~\cite{eldering2013normally}.

The notion of the convergence of immersions from pointed manifolds is defined as follows. It is the immersion map version of the Cheeger-Gromov convergence.

\begin{definition}\label{pointed:manifold}
Let $(M, g)$ be an  $n-$dimensional complete Riemannian manifold with bounded geometry. Assume that for each $k \geqslant 1$ we have an $m-$dimensional pointed manifold $\left(\Sigma_{k}, p_{k}\right)$ and an immersion $x_{k}: \Sigma_{k} \rightarrow M$. Then, we say that a sequence of immersions $\left\{x_{k}: \Sigma_{k} \rightarrow M\right\}_{k=1}^{\infty}$ converges to an immersion $x_{\infty}: \Sigma_{\infty} \rightarrow M$ from an $m-$dimensional pointed manifold $\left(\Sigma_{\infty}, p_{\infty}\right)$ if there exist
\begin{enumerate}
\item An exhaustion $\left\{U_{k}\right\}_{k=1}^{\infty}$ of $\Sigma_{\infty}$ with $p_{\infty} \in U_{k}$, and
\item A sequence of diffeomorphisms $\Psi_{k}: U_{k} \rightarrow V_{k} \subset \Sigma_{k}$ with $\Psi_{k}\left(p_{\infty}\right)=p_{k}$ such that the sequence of maps $x_{k} \circ \Psi_{k}: U_{k} \rightarrow M$ converges in $C^{\infty}$ to $x_{\infty}: \Sigma_{\infty} \rightarrow M$ uniformly on compact sets in $\Sigma_{\infty}$. 
\end{enumerate} 
\end{definition} 

In order to apply the Arzelá-Ascoli theorem, we need a type of uniform interior estimate for $|\widehat\nabla_g^k \mathcal A(\widetilde x_s)|$ on $\widetilde \Sigma_s\in\widetilde{\mathscr F}$ as done by Ecker and Huisken~\cite[Sect.~3]{ecker1991interior} for the MCF in Euclidean space and by the third author~\cite[Sect.~4]{yamamoto2020meancurvature} for the MCF in a Ricci flow background.

\begin{proposition}\label{Underthesamesatupofproposition01}
Assume that $(M,g)$ is an $n(\geqslant 3)-$dimensional Riemannian manifold with bounded geometry, and let $(\overline{g}(t),\overline{w}(t))$ be a shrinking self-similar solution to the extended Ricci flow satisfying \eqref{Additional-BG} on $M$ with potential function $\overline{f}$ and initial value $(g,w)$. Given an $(n-1)-$dimensional compact smooth manifold $\Sigma$ without boundary, and let $\mathscr F$ be the MCF of $\Sigma$ in a gradient shrinking extended Ricci soliton background which develops a singularity of type-I. Consider the  normalized MCF $\widetilde{\mathscr F}$ in $(M,g)$. Then, for every integer $k\geqslant0$ there is a positive constant $C_k$ which does not depend on $s$ such that
\begin{equation*}
|\widehat\nabla_g^k \mathcal A(\widetilde x_s)|_g \leqslant C_k \quad \mbox{on} \quad  \Sigma \times [-\log T, \infty),  
\end{equation*}
where $\widehat\nabla_g$ is defined from the Levi–Civita connection on $(\Sigma, \widetilde x^*_sg).$
\end{proposition}

\begin{proof}
The proof follows a standard approach as in \cite[Sect.~7]{huisken1990asymptotic}, \cite[Sect.~7]{Bernhard_List} and \cite[Prop.~4.9]{yamamoto2020meancurvature}, which is by induction on $k \in\mathbb N$. First of all, since $\overline g = (T - t)\psi^*_t g$ and $\widetilde x_s = \psi_t \circ x_t$, we have
\begin{align}
|\widehat\nabla_g^k\mathcal A(\widetilde x_s)|_g &= (T - t)^{\frac{1}{2} + \frac{1}{2}k}|\widehat\nabla_{\overline g}^k \mathcal  A_{\overline g}(x_t)|_{\overline g}, \label{(27)}\\
|\nabla_g^{k}\mathrm{Rm}_g|_g&=(T-t)^{1+\frac{1}{2}k}|\nabla_{\overline g}^{k}\mathrm{Rm}_{\overline g}|_{\overline g},\label{(28)}\\
|\nabla_g^{k}\big(dw \otimes dw\big)|_g&=(T-t)^{1+\frac{1}{2}k}|\nabla_{\overline g}^{k}\big(d\overline w \otimes d\overline w\big)|_{\overline g},\label{(28)*}
\end{align}
where $\nabla_g,$ $\widehat\nabla_g,$ $\nabla_{\overline g}$
and $\widehat\nabla_{\overline g}$ are defined from the Levi–Civita connection on $(M, g),$ $(\Sigma,  \widetilde x^*_s g),$ $(M, \overline g)$ and $(\Sigma, x^*_t\overline g)$, respectively.
Thus, the degree of $\widehat\nabla_{\overline g}^k\mathcal A_{\overline g}$ is $\frac{1}{2} +\frac{1}{2}k$ and of $\nabla_{\overline g}^k \operatorname{Rm}(\overline{g})$ and $\nabla_{\overline g}^k (d \overline w \otimes d \overline w)$ are $1 + \frac{1}{2}k$. Also, we write $\mathcal A_{\overline g}(x_t)$ and $\mathcal A_g(\widetilde x_s)$ by $\overline{\mathcal A}$ and $\mathcal A,$ respectively, while $\operatorname{Rm}(\overline g)$ and $\operatorname{Rm}(g)$ by $\overline{\operatorname{Rm}}$ and $\operatorname{Rm},$ respectively. 

Now, for two tensors $\mathcal T_1$ and $\mathcal T_2$, we write $\mathcal T_1*\mathcal T_2$ to mean a tensor formed by a sum of terms, each one of them obtained by contracting some indices of the pair $\mathcal T_1$ and $\mathcal T_2$ by using $g, x^*g$ and its inverses. In particular, there is a property that 
\begin{equation*}
|\mathcal T_1 \ast \mathcal T_2| \leqslant C|\mathcal T_1||\mathcal T_2|   
\end{equation*}
for some positive constant $C$ that depends only on the algebraic structure of $\mathcal T_1*\mathcal T_2$.
For $a, b \in \mathbb Q$, we consider $V_{a,b}$ the set of all (time-dependent) tensors $\mathcal T$ on $M$ that can be expressed as
\begin{eqnarray*}
\mathcal T &=& (\nabla^{k_{1}}_{\overline g}\overline{\mathrm{Rm}}\ast\dots\ast\nabla^{k_{I}}_{\overline g}\overline{\mathrm{Rm}})\ast \big(\nabla_{\overline g}^{m_{1}}(\alpha_n d\overline w \otimes d\overline w)\ast\dots\\
&&\ast\nabla_{\overline g}^{m_{K}}(\alpha_n d\overline w \otimes d\overline w)\big)\ast(\widehat\nabla_{\overline g}^{\ell_{1}}\overline{\mathcal  A}\ast\cdots\ast\widehat\nabla_{\overline g}^{\ell_{J}}\overline{\mathcal A})\ast (\mathop{\ast}^{p}Dx)
\end{eqnarray*}
with $I, J, K, p, k_1, \ldots , k_I , m_1, \ldots, m_K,  \ell_1, \ldots, \ell_J \in \mathbb N$ satisfying 
\begin{align*}
\sum^I_{i=1}\left( 1 + \frac{1}{2}k_i\right)+ \sum^K_{k=1}\left( 1 + \frac{1}{2}m_k\right)+ \sum^J_{j=1}\left(\frac{1}{2}+\frac{1}{2}\ell_j\right)= a \  \mbox{and} \ \sum^J_{j=1} \ell_j \leqslant b,
\end{align*}
and we define a vector space $\mathcal V_{a,b}$ as the set of all tensors $\mathcal T$ on $M$ which can be expressed as $\mathcal T = a_1\mathcal T_1 + \cdots + a_r\mathcal T_r$ for some $r \in \mathbb N, a_1 \ldots a_r \in \mathbb R$ and $\mathcal T_1, \ldots , \mathcal T_r \in V_{a,b}$.

The first step for the induction follows from the singularity of the type-I assumption \eqref{singularity:type:I}. For a fixed $k \geqslant 1$, assume that there exist positive constants $C_0, C_1, \ldots , C_{k-1}$ such that 
\begin{equation*}
|\widehat\nabla_g^i\mathcal A| \leqslant C_i \ \mbox{on}\  \Sigma \times [-\log T,\infty)   
\end{equation*}
for $i = 0, 1, \ldots , k-1$. We consider the evolution equation of $|\widehat \nabla_g^k\mathcal A|^2$, and finally, we will prove the bound of $|\widehat\nabla_g^k\mathcal A|^2$ by the parabolic maximum principle. From~\eqref{(27)} we have $| \widehat\nabla_g^k\mathcal A|^2 = (T-t)^{k+1}|\widehat\nabla^k_{\overline g}\overline{\mathcal A}|^2,$ and since $\frac{\partial}{\partial s} = (T-t)\frac{\partial}{\partial t},$ we obtain
\begin{align*}
\frac{\partial}{\partial s}|\widehat \nabla_g^k \mathcal  A|^2 &= -(k + 1)|\widehat \nabla_g^k \mathcal A|^2+(T-t)^{k+2} \frac{\partial}{\partial t}|\widehat\nabla_{\overline g}^k \overline{\mathcal A}|^2\leqslant (T-t)^{k+2} \frac{\partial}{\partial t}|\widehat\nabla_{\overline g}^k\overline{\mathcal A}|^2.    
\end{align*}
As used in the proof of~\cite[Prop.~4.9]{yamamoto2020meancurvature}, there exist tensors $\mathcal E[k] \in \mathcal V_{\frac{3}{2} + \frac{1}{2}k,k}$, $\mathcal C[k] \in \mathcal V_{\frac{3}{2} + \frac{1}{2}k,k+1}$ and $\mathcal G[k] \in \mathcal V_{\frac{1}{2} + \frac{1}{2}k,k-1}$ so that
\[\frac{\partial}{\partial t}|\widehat\nabla_{\overline g}^k \overline{\mathcal A}|^2 = \widehat\Delta_{\overline g}|\widehat\nabla_{\overline g}^k \overline{\mathcal A}|^2 - 2|\widehat\nabla_{\overline g}^{k+1} \overline{\mathcal A}|^2 + \mathcal E[k] * \widehat\nabla_{\overline g}^k \overline{\mathcal A} + \mathcal C[k] * \mathcal G[k],\]
where $\widehat\Delta_{\overline g}$ is the Laplacian on $\big(\Sigma, x^*_t\overline g(t)\big)$. Setting $\widehat\Delta_g$ for the Laplacian on $\big(\Sigma, \widetilde x^*_sg\big)$, we get  $(T - t)\widehat\Delta_{\overline g}=\widehat \Delta_g$, and then from~\eqref{(27)} one has
\begin{equation*}
(T-t)^{k+2}(\widehat\Delta_{\overline g}|\widehat\nabla_{\overline g}^k \overline{\mathcal A}|^2-2|\widehat\nabla_{\overline g}^{k+1} \overline{\mathcal A}|^2)= \widehat \Delta_g |\widehat \nabla_g^k \mathcal A|^2-2|\widehat\nabla_g^{k+1} \mathcal A|^2.    
\end{equation*}
Since $\mathcal G[k] \in \mathcal V_{\frac{1}{2} + \frac{1}{2}k,k-1}$, there exist $r \in \mathbb N$,  $a_1, \ldots, a_r \in \mathbb R$ and 
$\mathcal G[k]_1, \ldots , \mathcal G[k]_r \in  V_{\frac{1}{2} + \frac{1}{2}k,k-1}$
such that
$\mathcal G[k] = \sum_{i = 1}^r a_i\mathcal G[k]_i.$ Thus, $|\mathcal G[k]| \leqslant \sum_{i = 1}^r |a_i||\mathcal G[k]_i|.$
By definition of $V_{\frac{1}{2} + \frac{1}{2}k,k-1}$, each $\mathcal G[k]_i$ can be expressed as
\begin{align*}
&(\nabla^{k_{1}}_{\overline g}\overline{\mathrm{Rm}}\ast\dots\ast\nabla^{k_{I}}_{\overline g}\overline{\mathrm{Rm}})\ast \big(\nabla_{\overline g}^{m_{1}}(\alpha_n d\overline w \otimes d\overline w)\ast\dots\ast\nabla_{\overline g}^{m_{K}}(\alpha_n d\overline w \otimes d\overline w)\big)\\
&\ast(\widehat\nabla_{\overline g}^{\ell_{1}}\overline{\mathcal  A}\ast\cdots\ast\widehat\nabla_{\overline g}^{\ell_{J}}\overline{\mathcal A})\ast (\mathop{\ast}^{p}Dx)    
\end{align*}
with $I, J, K, p, k_1, \ldots , k_I , m_1, \ldots, m_K, \ell_1, \ldots , \ell_J \in \mathbb N$ satisfying
\[\sum^I_{i=1}\left(1 +\frac{1}{2}k_i\right)+ \sum^K_{k=1}\left( 1 + \frac{1}{2}m_k\right)+\sum^J_{j=1}\left(\frac{1}{2}+\frac{1}{2}\ell_j\right)=\frac{1}{2}+\frac{1}{2}k \ \mbox{and} \  \sum^J_{j=1}\ell_j \leqslant k - 1.\]
Hence, by using \eqref{(27)}, \eqref{(28)}, \eqref{(28)*} and $|Dx| =\sqrt{n - 1},$ we obtain
\begin{align*}
(T - t)^{\frac{1}{2} + \frac{1}{2}k}|\mathcal G[k]_i| \leqslant & C(T - t)^{\frac{1}{2} + \frac{1}{2}k}|\nabla_{\overline g}^{k_1} \overline{\operatorname{Rm}}| \cdots |\nabla_{\overline g}^{k_I} \overline{\operatorname{Rm}}||\nabla_{\overline g}^{m_{1}}(d\overline w\otimes d\overline w)|\cdots\\
&|\nabla_{\overline g}^{m_{K}}(d\overline w\otimes d\overline w)| |\widehat\nabla_{\overline g}^{\ell_1} \overline{\mathcal A}| \cdots | \widehat\nabla_{\overline g}^{\ell_J} \overline{\mathcal A}||Dx|^p \\
=& C\big(\sqrt{ n - 1}\big)^p|\nabla_g^{k_1} \operatorname{Rm}| \cdots |\nabla_g^{k_I} \operatorname{Rm}||\nabla^{m_{1}}(dw\otimes dw)|\cdots\\
&|\nabla^{m_{K}}(dw\otimes dw)|\widehat\nabla_g^{\ell_1} \mathcal A| \cdots |\widehat\nabla_g^{\ell_J} \mathcal A| 
\end{align*}
for some constant $C:= C(\alpha_n)>0$. Since $(M, g)$ has bounded geometry and \eqref{Additional-BG} holds, each $|\nabla_g^{k_i} \operatorname{Rm}|$ and $|\nabla_g^{m_k}(dw \otimes dw)|$ are uniformly bounded, and for $\ell_j \leqslant k-1$, each $|\widehat\nabla_g^{\ell_j}\mathcal A|$ is uniformly bounded by assumption of induction. So, there exists a constant $C'>0$ such that
\begin{equation*}
 (T-t)^{\frac{1}{2}+\frac{1}{2}k}|\mathcal G[k]| \leqslant C'.
\end{equation*}
In the same way, since $\mathcal E[k] \in \mathcal V_{\frac{3}{2}+\frac{1}{2}k,k}$, there exist $r'\in\mathbb N$, $b_1, \ldots, b_{r'}\in\mathbb R$ and $\mathcal E[k]_1, \ldots, \mathcal E[k]_{r'} \in V_{\frac{3}{2}+\frac{1}{2}k,k}$ such that $\mathcal E[k] = \sum_{i = 1}^{r'}b_i\mathcal E[k]_i.$
Consequently, $|\mathcal E[k]| \leqslant \sum_{i = 1}^{r'}  |b_i|\left|\mathcal E[k]_i\right|.$
By definition of $V_{\frac{3}{2}+\frac{1}{2}k,k}$, each $\mathcal E[k]_i$ can be expressed as
\begin{align*}
&(\nabla^{k_{1}}_{\overline g}\overline{\mathrm{Rm}}\ast\dots\ast\nabla^{k_{I}}_{\overline g}\overline{\mathrm{Rm}})\ast \big(\nabla_{\overline g}^{m_{1}}(\alpha_n d\overline w \otimes d\overline w)\ast\dots\ast\nabla_{\overline g}^{m_{K}}(\alpha_n d\overline w \otimes d\overline w)\big)\\
&\ast(\widehat\nabla_{\overline g}^{\ell_{1}}\overline{\mathcal  A}\ast\cdots\ast\widehat\nabla_{\overline g}^{\ell_{J}}\overline{\mathcal A})\ast (\mathop{\ast}^{p}Dx)    
\end{align*}
with $I, J, p, k_1, \cdots , k_I , \ell_1, \cdots, \ell_J \in \mathbb{N}$ satisfying
\[\sum^I_{i=1}\left(1 +\frac{1}{2}k_i\right)+ \sum^K_{k=1}\left( 1 + \frac{1}{2}m_k\right) + \sum^J_{j=1}\left(\frac{1}{2}+\frac{1}{2}\ell_j\right)=\frac{3}{2}+\frac{1}{2}k \ \mbox{and} \  \sum^J_{j=1}\ell_j \leqslant k.\]
If $\max\{\ell_1, \ldots , \ell_J \} \leqslant k-1$, we can prove that $(T-t)^{\frac{3}{2}+\frac{1}{2}k}|\mathcal E[k]_i|$ is bounded by the same argument as the case of $\mathcal G[k]_i$. If $\max\{\ell_1, \ldots , \ell_J \} = k$, one can see that the possible forms of $\mathcal E[k]_i$ are
\begin{equation*}
\begin{aligned}
\overline{\mathcal A} * \overline{\mathcal A} * \widehat\nabla_{\overline g}^k \overline{\mathcal A} * (\overset{p}{*} Dx), \\
\overline{\operatorname{Rm}} * \widehat\nabla_{\overline g}^k \overline{\mathcal A} * (\overset{p}{*} Dx),\\
d \overline w \otimes d \overline w * \widehat\nabla_{\overline g}^k \overline{\mathcal A} * (\overset{p}{*} Dx).
\end{aligned}
\end{equation*}
For each term, we can see by the same argument as the case of $\mathcal G [k]_i$ that there exists a constant $\widetilde C > 0$ such that $(T-t)^{\frac{3}{2}+\frac{1}{2}k}|\mathcal E[k]_i|\leqslant \widetilde C| \widehat\nabla_g^k \mathcal A|$. Hence, there exists a constant $C''>0$ such that
\begin{equation*}
(T-t)^{\frac{3}{2} + \frac{1}{2}k}|\mathcal E[k]| \leqslant C''(1 + |\widehat\nabla_g^{k} \mathcal A|).
\end{equation*}
Since $\mathcal C[k] \in \mathcal V_{\frac{3}{2}+\frac{1}{2}k,k+1}$, there exist $r''\in \mathbb N$, $c_1, \ldots, c_r'' \in \mathbb R$ and 
$\mathcal C[k]_1, \ldots ,\mathcal C[k]_{r''} \in  V_{\frac{3}{2} + \frac{1}{2}k,k+1}$ such that $\mathcal C[k] = \sum_{i = 1}^{r''} c_i\mathcal C[k]_i.$ Thus, $|\mathcal C[k]| \leqslant \sum_{i = 1}^{r''}  |c_i||\mathcal C[k]_i|.$ By definition of $V_{\frac{3}{2} + \frac{1}{2}k,k+1},$ each $\mathcal C[k]_i$ can be expressed as
\begin{align*}
&(\nabla^{k_{1}}_{\overline g}\overline{\mathrm{Rm}}\ast\dots\ast\nabla^{k_{I}}_{\overline g}\overline{\mathrm{Rm}})\ast \big(\nabla_{\overline g}^{m_{1}}(\alpha_n d\overline w \otimes d\overline w)\ast\dots\ast\nabla_{\overline g}^{m_{K}}(\alpha_n d\overline w \otimes d\overline w)\big)\\
&\ast(\widehat\nabla_{\overline g}^{\ell_{1}}\overline{\mathcal  A}\ast\cdots\ast\widehat\nabla_{\overline g}^{\ell_{J}}\overline{\mathcal A})\ast (\mathop{\ast}^{p}Dx)    
\end{align*}
with $I, J, K, p, k_1, \ldots , k_I , m_1, \ldots, m_K, \ell_1, \ldots , \ell_J \in \mathbb N$ satisfying 
\begin{equation*}
\sum^I_{i=1}\Big(1+\frac{1}{2}k_i\Big)+ \sum^K_{k=1}\left( 1 + \frac{1}{2}m_k\right) +\sum^J_{j=1}\Big(\frac{1}{2}+\frac{1}{2}\ell_j\Big)=\frac{3}{2}+\frac{1}{2}k \ \mbox{and} \  \sum^J_{j=1} \ell_j \leqslant k+ 1.    
\end{equation*}
If $\max \{\ell_1, . . . , \ell_J \}\leqslant k -1$, we can prove that $(T - t)^{\frac{3}{2} + \frac{1}{2}k}|\mathcal C[k]_i|$ is bounded by the same argument as the case of $\mathcal G[k]_i$ If $\max\{\ell_1, \ldots , \ell_J \} = k$, one can see that the possible forms of $\mathcal C[k]_i$ are
\begin{equation*}
\begin{aligned}
\overline{\mathcal A} * \overline{\mathcal A} * \widehat\nabla_{\overline g}^k \overline{\mathcal A} * (\overset{p}{*} Dx), \\
\overline{\operatorname{Rm}} * \widehat\nabla_{\overline g}^k \overline{\mathcal A} * (\overset{p}{*} Dx), \\
d \overline w \otimes d \overline w * \widehat\nabla_{\overline g}^k \overline{\mathcal A} * (\overset{p}{*} Dx).
\end{aligned}
\end{equation*}
Also, we get $(T -t)^{\frac{3}{2} + \frac{1}{2}k}$
$|\mathcal C[k]_i| \leqslant \widetilde C|\widehat\nabla_g^k \mathcal A|$ as the case of $\mathcal E[k]_i$. If $\max\{\ell_1, \ldots , \ell_J \} =k + 1$, one can easily see that the possible form of $\mathcal C[k]_i$ is
\[\widehat\nabla_{\overline g}^{k+1}\overline{\mathcal A} * (\overset{p}{*} Dx),\]
and $(T - t)^{\frac{3}{2} + \frac{1}{2}k}|\mathcal 
 C[k]_i| \leqslant \widetilde{C}' |\widehat\nabla_g^{k+1}\mathcal A|$ for some constant $\widetilde{C}' > 0$. Then, we
can see that there exists a constant $C''' > 0$ such that
\[(T - t)^{\frac{3}{2} + \frac{1}{2}k}|\mathcal C[k]| \leqslant C'''(1 + |\widehat\nabla_g^k \mathcal A| + |\widehat\nabla_g^{k+1} \mathcal A|).\]   
Thus, we have
\begin{align*}
\frac{\partial}{\partial s} |\widehat\nabla_g^k \mathcal A|^2 \leqslant &(T - t)^{k+2} \frac{\partial}{\partial t}|\widehat\nabla_{\overline g}^k \overline{\mathcal A}|^2\\
\leqslant & \widehat\Delta_g  |\widehat\nabla_g^k \mathcal A|^2 - 2|\widehat\nabla_g^{k+1}\mathcal A|^2 + C''(1 + |\widehat\nabla_g^k\mathcal A|)|\widehat\nabla_g^k\mathcal A|\\
&+ C'C'''(1 + |\widehat\nabla_g^k\mathcal A| + |\widehat\nabla_g^{k+1}\mathcal A|).    
\end{align*}
Since $-|\widehat\nabla_g^{k+1} \mathcal A|^2 + C'C'''|\widehat\nabla_g^{k+1}\mathcal A| \leqslant (C'C''')^2/4,$ we get
\begin{align*}
\frac{\partial}{\partial s} |\widehat\nabla_g^k \mathcal A|^2 &\leqslant \widehat\Delta_g |\widehat\nabla_g^k \mathcal A|^2 - |\widehat\nabla_g^{k+1} \mathcal A|^2\\
&\quad+ C''|\widehat\nabla_g^k \mathcal A|^2 + (C'' + C'C''')|\widehat\nabla_g^k \mathcal A| + C'C''' +(C'C''')^2/4.    
\end{align*}
By putting $\overline C_k := C'' + (C'' + C'C''') + C'C''' +(C'C''')^2/4$, we obtain
\begin{align}\label{eq:30:Yam}
\frac{\partial}{\partial s} |\widehat\nabla_g^k \mathcal A|^2 \leqslant \widehat\Delta_g|\widehat\nabla_g^k \mathcal A|^2 - |\widehat\nabla_g^{k+1}\mathcal A|^2 + \overline C_k(1 + |\widehat\nabla_g^k \mathcal A|^2).    
\end{align}
So,
\begin{align}\label{eq:31:Yam}
\frac{\partial}{\partial s}|\widehat\nabla_g^k \mathcal A|^2 \leqslant \widehat\Delta_g|\widehat\nabla_g^k \mathcal A|^2 + \overline C_k(1 + |\widehat\nabla_g^k \mathcal A|^2 ).
\end{align}
We observe that inequality~\eqref{eq:30:Yam} also holds for $k - 1,$ that is, 
\begin{align}\label{eq:32:Yam}
\frac{\partial}{\partial s} |\widehat\nabla_g^{k-1}\mathcal A|^2 \leqslant \widehat\Delta_g|\widehat\nabla_g^{k-1} \mathcal A|^2 - |\widehat\nabla_g^k \mathcal A|^2 + \overline C_{k-1}(1 + |\widehat\nabla_g^{k-1} \mathcal A|^2),
\end{align}
for some constant $\overline C_{k-1} > 0.$ By combining inequalities~\eqref{eq:31:Yam} and~\eqref{eq:32:Yam}, one has
\begin{align*}
&\frac{\partial}{\partial s} (|\widehat\nabla_g^k \mathcal A|^2 + 2\overline C_k|\widehat\nabla_g^{k-1}\mathcal A|^2)\\
&\leqslant\widehat\Delta_g( |\widehat\nabla_g^k \mathcal A|^2 + 2\overline C_k|\widehat\nabla_g^{k-1} \mathcal A|^2) + \overline C_k - \overline C_k|\widehat\nabla_g^k \mathcal A|^2 + 2\overline C_k \overline C_{k-1}(1 + |\widehat\nabla_g^{k-1} \mathcal A|^2).  
\end{align*}
Since 
\begin{align*}
&\overline{C}_k - \overline{C}_k|\widehat\nabla_g^k\mathcal A|^2 + 2\overline C_k \overline C_{k-1}(1 + |\widehat\nabla_g^{k-1}\mathcal A|^2)\\
&= - \overline C_k(|\widehat\nabla_g^k \mathcal A|^2 + 2\overline C_k|\widehat\nabla_g^{k-1} \mathcal A|^2)  + \overline C_k(1 + 2\overline C_{k-1} + 2(\overline C_k + \overline C_{k-1})|\widehat\nabla_g^{k-1} \mathcal A|^2)    
\end{align*}
and $|\widehat\nabla_g^{k-1} \mathcal A|^2$ is bounded by assumption of induction, one can see that there exists a constant $\overline{\overline{C}}_k > 0$ such that
\begin{align*}
&\frac{\partial}{\partial s} (|\widehat\nabla_g^k \mathcal  A|^2 + 2\overline C_k|\widehat\nabla_g^{k-1}\mathcal A|^2 - \overline{\overline{C}}_k) \\
&\leqslant \widehat\Delta_g( |\widehat\nabla_g^k \mathcal A|^2+2\overline{C}_k|\widehat\nabla_g^{k-1} \mathcal A|^2 - \overline{\overline{C}}_k)- \overline{C}_k(|\widehat\nabla_g^k \mathcal A|^2 + 2\overline{C}_k|\widehat\nabla_g^{k-1} \mathcal A|^2 - \overline{\overline{C}}_k).    
\end{align*}
Thus, by putting $\mu := e^{\overline{C}_ks}(|\widehat\nabla_g^k \mathcal  A|^2 + 2\overline{C}_k| \widehat\nabla_g^{k-1} \mathcal  A|^2 - \overline{\overline{C}}_k),$ we get
$$\frac{\partial}{\partial s} \mu \leqslant \widehat\Delta_g \mu.$$
Since $\Sigma$ is compact, $\mu$ is bounded at the initial time $s = -\log T.$  Then, by the parabolic maximum principle, it follows that $\mu$ is also bounded on $\Sigma \times[-\log T,\infty)$, that is, there exists a constant $\widetilde C_k > 0$ such that $\mu \leqslant \widetilde C_k$ on $\Sigma \times [-\log T, \infty).$ So,
\begin{align*}
|\widehat\nabla_g^k \mathcal A|^2\leqslant e^{-\overline{C}_k s}\widetilde  C_k - 2\overline{C}_k|\widehat\nabla_g^{k-1}\mathcal A|^2 + \overline{\overline{C}}_k\leqslant C_k.    
\end{align*}
where $C_k := T^{\overline C_k} \widetilde C_k +  \overline{\overline{C}}_k.$
\end{proof}

The proof of Theorem~\ref{application:Huisken} also relies on the following lemmas. Before proving these lemmas, we observe that $S:= R - \alpha_n |\nabla w|^2$ is nonnegative along the gradient shrinking extended Ricci soliton on $M$. For it,  one uses $\frac{\partial S}{\partial t} = \Delta S + 2|\operatorname{Ric} - \alpha_n dw \otimes dw|^2 +2 \alpha_n (\Delta w)^2$ (see~\cite[Lem.~3.2]{Bernhard_List}) and $|T|^2 \geqslant \frac{(\operatorname{tr} T)^2}{n}$ for $T=\operatorname{Ric} - \alpha_n dw \otimes dw$ to obtain
\begin{equation*}
\frac{\partial S}{\partial t} \geqslant \Delta S + \frac{2}{n} S^2. 
\end{equation*}
Now, from maximum principle and since $t \in [-\infty, T)$ we can prove the required result on $S$ (see~\cite[Lem.~2.18]{chow2006hamilton}, for further details). Besides,  we can assume 
\begin{equation}\label{like Hamilton Identity}
S_{\overline g} + |\nabla \overline f|^2 - \frac{\overline f}{T - t}  = 0
\end{equation}
along the gradient extended Ricci soliton on $M$. For it, we proceed as in the proof of Hamilton's equation (see Cao~\cite{HDCao} or Hamilton~\cite{Hamilton}) to show that $S_g+|\nabla f|^2-f= C$, for some constant $C$. Indeed, from \eqref{mod_grad_Ricci_soliton} at $t = T  - 1$ and well known facts in the literature, we obtain
\begin{eqnarray*}
0&=&\operatorname{div}_g \operatorname{Ric}_g+ \operatorname{div}_g(\nabla d f) - \alpha_n \operatorname{div}_g(d w \otimes d w)\\
&=&\frac{dR_g}{2} + d\Delta_g f + \operatorname{Ric}_g(\nabla f, \cdot)  - \alpha_n(\Delta_g w d w + \frac{1}{2}d|\nabla w|^2) \\
&=&\frac{dR_g}{2} +  d\Delta_g f + \operatorname{Ric}_g(\nabla f, \cdot)  - \alpha_n g(\nabla w, \nabla f)d  w - \frac{\alpha_n}{2}d|\nabla w|^2 \\
&=&\frac{dR_g}{2} +  d\left(\frac{n}{2} - R_g + \alpha_n |\nabla  w|^2\right) + \frac{1}{2}d f -\frac{1}{2} d|\nabla f|^2 \\
&&+ \alpha_n  g( \nabla w, \nabla  f) d w - \alpha_n g(\nabla  w, \nabla f)d  w - \frac{\alpha_n}{2}d|\nabla w|^2 \\
&=&-\frac{1}{2}dR_g +\frac{1}{2} \alpha_n d|\nabla w|^2 + \frac{d f}{2} -\frac{1}{2} d |\nabla  f|^2,
\end{eqnarray*}
whence $S_g+|\nabla f|^2-f= C$, for some constant $C$. Thus, we can assume, adding $C$ to $f$ (if necessary), $S_g+|\nabla f|^2-f=0$. Since $\bar g(t)=\sigma(t)\psi_t^*g$, where $\sigma(t)=T-t$, then $S_{\psi_t^*g}+|\nabla_{\psi_t^*g} \bar f|^2-\bar f=0$. By conformal theory, $R_{\bar{g}}=R_{\psi_t^*g}/\sigma(t)$ and $|\nabla_{\bar g} \bar w|^2=|\nabla_{\psi_t^*g} \bar w|^2/\sigma(t)$ (the same to $\bar f$), and then we obtain \eqref{like Hamilton Identity}.

The following lemma gives the variation of the weighted area functional on the normalized MCF.

\begin{lemma}\label{Prop:4.6}
Assume that $(M,g)$ is an $n(\geqslant 3)-$dimensional Riemannian manifold, and let $(\overline{g}(t),\overline{w}(t))$ be a shrinking self-similar solution to the extended Ricci flow on $M$ with potential function $\overline{f}$ and initial value $(g,w)$. Given an $(n-1)-$dimensional compact smooth manifold $\Sigma$ without boundary, and let $\mathscr F$ be the MCF of $\Sigma$ in a gradient shrinking extended Ricci soliton background. Consider the  normalized MCF $\widetilde{\mathscr F}$ in $(M,g)$. Then
\begin{equation*}
\frac{d}{ds}\int_\Sigma e^{-f \circ \widetilde x_s}dA_{\widetilde x_s^*g} = -\int_\Sigma \big( H_g(\widetilde x_s) + e_s  (f \circ \widetilde x_s) \big)^2  e^{-f \circ \widetilde x_s}dA_{\widetilde x_s^*g}, \  s \in [-\log T , \infty).
\end{equation*}
\end{lemma}

\begin{proof}
As $\overline f\circ x_t =  f \circ \widetilde x_s$ and $x^*_t \overline g= (T-t)  (\psi_t \circ x_t)^* g = (T-t)\widetilde x^*_s g$ both on $\Sigma$, we have
\begin{align}
[4\pi(T - t)]^{-\frac{n - 1}{2}} e^{-\overline f} dA_{\overline g}&=(4\pi)^{-\frac{n -1}{2}} e^{-f\circ \widetilde x_s} dA_{\widetilde x_s^*g},\label{xa:era:01}\\
(T - t)\left(H_{\overline g}+ e_t \overline  f \right)^2 &= \left(H_g(\widetilde x_s) + e_s (f \circ \widetilde x_s) \right)^2.\label{xa:era:02} 
\end{align}
The result of the lemma follows from Theorem~\ref{Huisken_monotonicity}, identities \eqref{xa:era:01} and \eqref{xa:era:02} together with the chain rule.
\end{proof}

\begin{lemma}\label{prop:4.7}
Assume that $(M,g)$ is an $n(\geqslant 3)-$dimensional Riemannian manifold with bounded geometry, and let $(\overline{g}(t),\overline{w}(t))$ be a shrinking self-similar solution to the extended Ricci flow satisfying \eqref{Additional-BG} on $M$ with potential function $\overline{f}$ and initial value $(g,w)$. Given an $(n-1)-$dimensional compact smooth manifold $\Sigma$ without boundary, and let $\mathscr F$ be the MCF of $\Sigma$ in a gradient shrinking extended Ricci soliton background which develops a singularity of type-I. Consider the  normalized MCF $\widetilde{\mathscr F}$ in $(M,g)$. Then, there exists a positive constant $C$ such that
\begin{equation*}
\int_\Sigma e^{-\frac{f}{2} \circ \widetilde x_s} dA_{\widetilde x^*_s g}\leqslant C   
\end{equation*}
uniformly on $[-\log T, \infty).$
\end{lemma}

\begin{proof}
We start by substituting $\bar{\rho}_{t}:=\frac{1}{[4\pi(T-t)]^{\frac{n}{2}}}e^{-\frac{\overline f}{2}}$, $u_{t}:=[4\pi(T-t)]^{1/2},$ $h:=-2\operatorname{Ric}_{\overline g}+2\alpha_n d \overline w \otimes \overline w$ and $V:= H_{\overline g}e_t$ into Prop.~3.2 of \cite{yamamoto2020meancurvature} to obtain 
\begin{eqnarray*}
&&\frac{d}{dt} \int_{\Sigma}u_{t}\mathop{x_{t}^{*}\hspace{-0.5mm}\bar{\rho}_{t}} dA_{\overline g}\\
&=&-\int_{\Sigma}u_{t}\Bigl(H_{\overline g}+\frac{1}{2} e_t\overline f\Bigr)^2\mathop{x_{t}^{*}\hspace{-0.5mm}\bar{\rho}_{t}} dA_{\overline g}\\
&&+\int_{\Sigma}u_{t} x_{t}^{*}\left(\Delta_{\overline g}\overline \rho_{t}+\frac{\partial \overline \rho_t}{\partial t}-\overline \rho_tS_{\overline g}\right)  dA_{\overline g}\\
&& +\int_{\Sigma}\left( \frac{\partial u_{t}}{\partial t} -\widehat\Delta_{\overline g}u_{t}+
u_{t}\left(\frac{1}{2}\mathrm{Hess}_{\overline g}\overline f-\frac{h}{2}\right)(e_{t},e_{t})\right)\mathop{x_{t}^{*}\hspace{-0.5mm}\bar{\rho}_{t}} dA_{\overline g}. 
\end{eqnarray*}
By using $\Delta_{\overline g} \overline f=-S_{\overline g}+\frac{n}{2(T-t)}$ (see \eqref{mod_grad_Ricci_soliton}), $|\nabla \overline f|^{2}=\frac{\overline f}{T-t}-S_{\overline g}$ (see \eqref{like Hamilton Identity}), $ S_{\overline g} \geqslant 0$ and \eqref{self-solution}, we obtain 
\begin{equation*}
\Delta_{\overline g}\overline \rho_{t}+\frac{\partial \overline \rho_t}{\partial t}-\overline\rho_tS_{\overline g}
={\overline\rho_{t}}\biggl(-\frac{\overline f}{4(T-t)}-\frac{S_{\overline g}}{4}+\frac{n}{4(T-t)} \biggr)\leqslant \frac{\overline \rho_t}{4(T-t)}(n-f).
\end{equation*}
Furthermore, since $u$ satisfies 
\begin{equation*}
\frac{\partial u_{t}}{\partial t} -\widehat\Delta_{\overline g}u_{t}+u_{t}\left(\mathrm{Hess}_{\overline g}\overline f-\frac{h}{2}\right)(e_{t},e_{t})=0,
\end{equation*}
we get
\begin{equation*}
\frac{\partial u_{t}}{\partial t} -\widehat\Delta_{\overline g}u_{t}+u_{t}\left(\frac{1}{2}\mathrm{Hess}_{\overline g}\overline f-\frac{h}{2}\right)(e_{t},e_{t})=-\frac{1}{2}u_{t}\mathrm{Hess}_{\overline g}\overline f(e_{t},e_{t}).
\end{equation*}
Equation $\mathrm{Hess}_{\overline g}\overline f=\frac{1}{2(T-t)}\overline g-\mathrm{Ric}_{\overline g} + \alpha_n d \overline w \otimes d\overline w$ implies 
\begin{equation*}
-\frac{1}{2}u_{t}\mathrm{Hess}_{\overline g}\overline f(e_{t},e_{t})=u_{t}\left(-\frac{1}{4(T-t)}+\frac{1}{2}(\mathrm{Ric}_{\overline g} - \alpha_n d\overline w \otimes d \overline w)(e_{t},e_{t})\right).
\end{equation*}
On the other hand, by assumption $dw \otimes dw$ is bounded (take $j=0$ in \eqref{Additional-BG}) and since $(M,g)$ has bounded geometry, we have that $C'':=\max_{M}\big\{|\mathrm{Ric}_g|_g + \alpha_n |dw \otimes dw|_g\big\}$ is a constant, then 
\begin{eqnarray*}
&&(\mathrm{Ric}_{\overline g}-\alpha_n d\overline w \otimes d\overline w)(e_{t},e_{t})\\
&\leqslant& |\mathrm{Ric}_{\overline g} -\alpha_n d\overline w \otimes d\overline w|_{\overline g}\leqslant\frac{|\mathrm{Ric}_g|_g + \alpha_n |dw \otimes dw|_g}{T-t}\leqslant \frac{C''}{T-t}.    
\end{eqnarray*}
Hence,
\begin{equation*}
\frac{d}{dt} \int_{\Sigma}u_{t}\mathop{x_{t}^{*}\hspace{-0.5mm}\bar{\rho}_{t}} dA_{\overline g}
< \frac{1}{4(T-t)}\int_{\Sigma}\Bigl(C_{0}-\overline f\circ x_{t}\Bigr)u_{t} \mathop{x_{t}^{*}\hspace{-0.5mm}\bar{\rho}_{t}} dA_{\overline g}, 
\end{equation*}
where $C_{0}:=n+2C''$. Since $s=-\log(T-t),$ we have
\begin{eqnarray*}
\frac{d}{ds}\int_\Sigma e^{-\frac{f}{2}\circ\widetilde x_s}dA_{\widetilde x_s^*g}
&=&(4\pi)^{\frac{n-1}{2}}(T-t)\frac{d}{dt}\int_\Sigma u_{t}\mathop{x_{t}^{*}\hspace{-0.5mm}\bar{\rho}_{t}} dA_{\overline g}\\
&<&\frac{1}{4}\int_\Sigma \Bigl(C_{0}-f\circ \widetilde x_{s}\Bigr)e^{-\frac{f}{2}\circ\widetilde x_s}dA_{\widetilde x_s^*g}. 
\end{eqnarray*}
The result of the lemma follows from the analysis of the sign on the previous inequality.
\end{proof}

\begin{lemma}\label{Lem:4.10}
Assume that $(M,g)$ is an $n(\geqslant 3)-$dimensional Riemannian manifold with bounded geometry, and let $(\overline{g}(t),\overline{w}(t))$ be a shrinking self-similar solution to the extended Ricci flow satisfying \eqref{Additional-BG} on $M$ with potential function $\overline{f}$ and initial value $(g,w)$. Given an $(n-1)-$dimensional compact smooth manifold $\Sigma$ without boundary, and let $\mathscr F$ be the MCF of $\Sigma$ in a gradient shrinking extended Ricci soliton background which develops a singularity of type-I. Consider the  normalized MCF  $\widetilde{\mathscr F}$ in $(M,g)$. Then, there exists a constant $C' > 0$ such that
\begin{equation*}
\left|\frac{d^2}{d^2s}\int_\Sigma e^{-f\circ \widetilde x_s}dA_{\widetilde x^*_s g}\right|  = \left|\frac{d}{ds}\int_\Sigma \big(H_g(\widetilde x_s) + e_s  (f \circ \widetilde x_s) \big)^2e^{-f\circ \widetilde x_s}dA_{\widetilde x_s^*g}\right|\leqslant C'   
\end{equation*}
uniformly on $[-\log T, \infty).$ 
\end{lemma}
\begin{proof}
We already know that we can assume $ S+|\nabla f|^2-f=0 $ and $S\geq0$ along the gradient shrinking extended Ricci soliton on $M$. So, 
$0\leqslant |\nabla f|^2 \leqslant f$ and $\quad 0\leqslant S \leqslant f$. The result of the lemma follows from Lemma~\ref{prop:4.7} and the same steps as done in~\cite{yamamoto2020meancurvature}.
\end{proof}

We are now ready to prove the main result of this note.

\begin{proof}[\bf Proof of Theorem~\ref{application:Huisken}]\label{test}
We will make this proof in a way that allows a generalization to the more general context of bounded geometry and satisfying \eqref{Additional-BG}. Note that these hypotheses for the compact case are automatically satisfied. 

Take a sequence $\{\widetilde \Sigma_{s_j}\}_{j=1}^\infty$ in $\widetilde{\mathscr F}$ and points $\{p_j\}_{j=1}^\infty$ in $\Sigma$, and denote the Riemann curvature tensor of each $\widetilde \Sigma_{s_j}$ by $\widehat{\operatorname{Rm}}(\widetilde x_{s_j}^* g)$. 

Firstly, we will show that the sequence of pointed manifolds $\{(\Sigma,\widetilde{x}^*_{s_j}g,p_j)\}$ converges to some complete pointed Riemannian manifold $(\Sigma_{\infty}, h_\infty, p_\infty)$  in the $C^\infty$ Cheeger-Gromov sense up to subsequence. 

Since $(M,g)$ has bounded geometry, there exist positive constants $D_p$ and $\eta$ such that
\begin{equation}\label{Aux-GEq-WL1}
|\nabla^p \operatorname{Rm}_g|
\leqslant D_p \quad \mbox{and} \quad \operatorname{inj}(M,  g) \geqslant \eta > 0.
\end{equation}
for every integer $p\geqslant 0$. Besides, since assumption \eqref{Additional-BG} holds, then by Proposition~\ref{Underthesamesatupofproposition01} there are positive constants $C_p$ which does not depend on $s_j$ such that
\begin{equation}\label{Aux-GEq-WL2}
|\widehat\nabla_g^p \mathcal A(\widetilde x_{s_j})| \leqslant C_p.
\end{equation}
Thus, we are able to apply Thm.~2.1 by Chen and Yin \cite{chen2007uniqueness} which guarantees the existence of a positive constant $\delta = \delta(C_0, D_0, \eta, n)$ such that the injectivity radius of each $(\Sigma,\widetilde{x}^*_{s_j}g)$ satisfies
\begin{equation*}
\operatorname{inj}(\Sigma,\widetilde{x}^*_{s_j}g) \geqslant \delta > 0.
\end{equation*}
Moreover, from \eqref{Aux-GEq-WL1}, \eqref{Aux-GEq-WL2}, the Gauss equation and its iterated derivatives, we obtain positive constants $\widetilde C_p$ which also do not depend on $s_j$ such that
\begin{equation*}
|\widehat\nabla^p \widehat{\operatorname{Rm}}(\widetilde x_{s_j}^* g)| \leqslant \widetilde C_p
\end{equation*}
for every integer $p\geqslant0$. Hence, by Arzelà-Ascoli theorem, there exists a subsequence $\{(\Sigma, \widetilde  x_{s_{j_k}}^* g, p_{j_k})\}$ which converges to some complete pointed Riemannian manifold $(\Sigma_\infty, h_\infty, p_\infty)$, i.e., there exist an exhaustion $\{U_{j_k}\}_{k=1}^\infty$ of $\Sigma_{\infty}$ with $p_{\infty} \in U_{j_k}$ and diffeomorphisms $\Psi_{j_k}:U_{j_k}\to V_{j_k}:= \Psi_{j_k}(U_{j_k}) \subset \Sigma$ with $\Psi_{j_k}(p_\infty) = p_{j_k}$ such that $\Psi^*_{j_k}(\widetilde x^*_{s_{j_k}} g)$ converges in $C^{\infty}$ to $h_\infty$ uniformly on compact sets in $\Sigma_{\infty}$.

Secondly, we will prove the existence of an immersion map $x_{\infty}: \Sigma_\infty\to (M,g)$. For this, let $\Theta: (M, g) \to (\mathbb R^d, g_{\rm st})$ be a Nash isometric embedding in some higher dimensional Euclidean space such that, for every integer $j \geqslant 0$, the norm $|\nabla_g^{j} \mathcal A(\Theta)| \leqslant \overline D_{j}$, for some constants $\overline D_{j}>0$, where $\mathcal A(\Theta)$ is the second fundamental form of $\Theta$.

\emph{At this point of the proof, we are assuming that the sequence $\{\widetilde x_{s_{j_k}}(p_{j_k})\}$ and the norms $|\nabla_g^{j} \mathcal A(\Theta)|$ are uniformly bounded in $M$, which are actually true in the compact case. Note, however, that we will need to prove or assume these facts for the noncompact case.}

Since we are working in the context of bounded geometry and each norm $|\nabla_g^{j} \mathcal A(\Theta)|$ is uniformly bounded, by setting 
\begin{equation*}
\overline x_{s_{j_k}} := \Theta \circ \widetilde x_{s_{j_k}} \circ \Psi_{j_k} :U_{j_k} \to (\mathbb R^d, g_{\rm st}),
\end{equation*}
we can prove that $|\nabla^p \overline x_{s_{j_k}}| \leqslant C_p$ for some positive constants $ C_p$ which does not depend on $s_{j_k}$, for every integer $p\geqslant0.$ The proof of this fact is by induction and follows the same steps as done in \cite[Appendix~C, pg.~42]{yamamoto2015meancurvature}. 

Then, by a standard argument as the Arzelà-Ascoli theorem, there exists a smooth map $\overline x_\infty: \Sigma_{\infty} \rightarrow (\mathbb R^d, g_{\rm st})$ such that the sequence of immersions $\overline x_{s_{j_k}}: U_{j_k} \to (\mathbb R^d,g_{\rm st})$ converges to $\overline x_{\infty}: \Sigma_{\infty}\rightarrow (\mathbb R^d,g_{\rm st})$ up to subsequence. By definition of $C^{\infty}$ convergence, we get $h_\infty =\overline x^*_\infty g_{\rm st}$ on $\Sigma_\infty$. This implies that $\overline x_\infty: \Sigma_{\infty} \rightarrow \mathbb R^d$ is an isometric immersion map and $x_\infty := \Theta^{-1} \circ \overline x_{\infty}: \Sigma_\infty \to (M, g)$ is the required immersion of the theorem with $x_\infty^*g = (\Theta^{-1} \circ \overline x_{\infty})^*\Theta^*g_{\rm st} = h_\infty$.

Next, we will prove that $\Sigma_{\infty}$ is a $f_{\infty}$-minimal hypersurface of $(M,g),$ where $f_\infty = f \circ x_\infty$. To simplify the notation we can take an exhaustion $\{U_k\}_{k=1}^\infty$ of $\Sigma_{\infty}$ with $p_{\infty} \in U_k$ and diffeomorphisms $\Psi_k : U_k \to V_k := \Psi_k(U_k) \subset \Sigma$ with $\Psi_k(p_\infty) = p_k$ such that $\Psi^*_k(\widetilde x^*_{s_{k}} g)$ converges in $C^{\infty}$ to $x^*_\infty  g$ uniformly on compact sets in $\Sigma_{\infty}$. Furthermore, the sequence of maps $\widetilde x_{s_{k}} \circ\Psi_k: U_k \rightarrow (M,g)$ converges in $C^{\infty}$ to $x_\infty: \Sigma_{\infty} \rightarrow(M,g)$ uniformly on compact sets in $\Sigma_{\infty}.$ So, for any compact set $K \subset \Sigma_\infty$ there exists $k_0$ such that $K \subset U_k$, for all $k \geqslant k_0$, and $\widetilde x_{s_{k}}\circ\Psi_k :U_k \rightarrow (M,g)$ converges to $x_\infty : \Sigma_\infty \to (M,g)$ in $C^\infty$ uniformly on $K$. Thus,
\begin{eqnarray*}
&&\int_K \left[H(\widetilde x_{s_{k}} \circ \Psi_k) + e_k\left(f\circ (\widetilde x_{s_{k}}\circ \Psi_k)\right) \right]^2 e^{- f\circ (\widetilde x_{s_{k}}\circ \Psi_k)} dA_{(\widetilde x_{s_{k}} \circ \Psi_k)^*g}\nonumber\\
&&\to \int_K \left( H( x_\infty) + e_\infty f_\infty\right)^ 2e^{- f_\infty }dA_{x^*_\infty g} 
\end{eqnarray*}
as $k\to \infty$, and 
\begin{eqnarray*}
&&\int_K \left[ H(\widetilde x_{s_{k}} \circ \Psi_k) +  e_k f (\widetilde x_{s_{k}}\circ\Psi_k)\right]^2e^{-f\circ (\widetilde x_{s_{k}}\circ \Psi_k)}dA_{(\widetilde x_{s_{k}} \circ\Psi_k)^* g}\\
&=&\int_{\Psi_k(K)}\left[H(\widetilde x_{s_{k}}) + e_k  (f\circ \widetilde x_{s_{k}}) \right]^2 e^{- f\circ \widetilde x_{s_{k}}} dA_{\widetilde x^*_{s_{k}}  g}\\
&\leqslant& \int_{\Sigma} \left[ H(\widetilde x_{s_{k}}) + e_k (f \circ \widetilde x_{s_{k}})\right]^2 e^{- f \circ\widetilde x_{s_{k}}} dA_{\widetilde x^*_{s_{k}}  g}.
\end{eqnarray*}
Hence, it is enough to prove the following: 
\begin{equation}\label{lastconv3}
\int_\Sigma [H(\widetilde x_{s_{k}})+ e_k( f \circ \widetilde x_{s_{k}})]^2 e^{-f\circ\widetilde x_{s_{k}}} dA_{\widetilde x_{s_{k}}^*g}\to 0
\end{equation}
as $k\to \infty$. We will argue by contradiction. Assume that there exist a constant $\delta>0$ and a subsequence $\{\ell\}\subset \{k\}$ with $\ell\to\infty$ such that 
\begin{equation*}
\int_\Sigma [H(\widetilde x_{s_{\ell}})+e_\ell(f \circ \widetilde x_{s_{\ell}})]^2 e^{-f\circ\widetilde x_{s_{\ell}}} dA_{\widetilde x_{s_{\ell}}^*g}\geqslant\delta.
\end{equation*}
Then
\begin{equation*}
\int_\Sigma [H(\widetilde x_s)+ e_s(f \circ \widetilde x_s)]^2 e^{-f\circ\widetilde x_s} dA_{\widetilde x_s^*g} \geqslant \frac{\delta}{2}, 
\end{equation*}
for $s\in[s_{\ell},s_{\ell}+\frac{\delta}{2C'}]$, 
where we used Lemma~\ref{Lem:4.10} and $C'$ is the constant appeared in that lemma. Hence, 
\begin{equation*}
\int_{-\log T}^\infty \int_\Sigma [H(\widetilde x_s)+ e_s(f \circ \widetilde x_s)]^2 e^{-f\circ\widetilde x_s} dA_{\widetilde x_s^*g} ds =\infty. 
\end{equation*}
On the other hand, by the monotonicity formula in Lemma~\ref{Prop:4.6}
\begin{equation*}
\frac{d}{ds} \int_\Sigma e^{-f\circ\widetilde x_s} dA_{\widetilde x_s^*g}=-\int_\Sigma [H(\widetilde x_s)+ e_s(f \circ \widetilde x_s)]^2 e^{-f\circ\widetilde x_s} dA_{\widetilde x_s^*g}\leqslant 0. 
\end{equation*}
Thus, the weighted volume 
\begin{equation*}
\int_\Sigma e^{-f\circ\widetilde x_s} dA_{\widetilde x_s^*g} 
\end{equation*}
is monotone decreasing and nonnegative. Therefore, it converges to some value 
\begin{equation*}
\alpha:=\lim_{s\to\infty}\int_\Sigma e^{-f\circ\widetilde x_s} dA_{\widetilde x_s^*g}<\infty, 
\end{equation*}
and then we obtain the following contradiction:
\begin{equation*}
\int_{-\log T}^\infty \int_\Sigma [H(\widetilde x_s)+ e_s(f \circ \widetilde x_s)]^2 e^{-f\circ\widetilde x_s} dA_{\widetilde x_s^*g}
=-\alpha +\int_\Sigma e^{-f\circ\widetilde x_a} dA_{\widetilde x_a^*g}<\infty,  
\end{equation*}
where $a:=-\log T$, which proves \eqref{lastconv3}. With this, the proof of the theorem is complete.
\end{proof}

\section{The noncompact case}\label{Sec-ncompcase}
In this section, we address the case of complete noncompact Riemannian manifolds $(M,g)$ with some additional uniformity conditions. We begin with a brief discussion on the reduced distance along the extended Ricci flow, which is a particular case defined by Müller for the context of Ricci Harmonic flow (see~\cite[Sect.~8]{muller2012ricci}), initially defined by Perelman in the Ricci flow setting (see~\cite[Sect.~7]{perelman2002entropy}).  

Let $(\overline g(t), \overline w(t))$ be a shrinking self-similar solution of the extended Ricci flow in $M \times [0, T)$.  For any smooth curve $\gamma:[t_{1},t_{2}]\to M$ with $0\leqslant t_{1}<t_{2}<T$, consider the $\mathcal{L}$-length of $\gamma$ by 
\begin{equation*}
\mathcal{L}(\gamma):=\int_{t_{1}}^{t_{2}}\sqrt{t_{2}-t}\bigl( S_{\overline g}+|\dot{\gamma}|^2 \bigr)dt, 
\end{equation*}
where $|\dot{\gamma}|$ is the norm of $\dot{\gamma}(t)$ measured by $\overline g$ and $S_{\overline g}= R_{\overline g} - \alpha_n |\nabla w|_{\overline g}^2$. 
For a fixed point $(q_{2},t_{2})$ in the space-time $M\times [0,T)$, Müller defined the reduced distance 
\begin{equation*}
\ell_{q_{2},t_{2}}:M\times [0,t_{2})\to\mathbb{R}
\end{equation*}
based at $(q_{2},t_{2})$ by
\begin{equation*}
\ell_{q_{2},t_{2}}(q_{1},t_{1}):=\frac{1}{2\sqrt{t_{2}-t_{1}}}\inf_{\gamma}\mathcal{L}(\gamma), 
\end{equation*}
where the infimum is taken over all smooth curve $\gamma:[t_{1},t_{2}]\to M$ with $\gamma(t_{1})=q_{1}$ and $\gamma(t_{2})=q_{2}$. 

In what follows, we assume that there exists a Nash isometric embedding $\Theta: (M, g) \to (\mathbb R^d, g_{\rm st})$ in some higher-dimensional Euclidean space such that, 
for every integer $j \geqslant 0,$ the second fundamental form $\mathcal A(\Theta)$ of $\Theta$ satisfies
\begin{equation}\label{CAdd-Thm3}
|\nabla_g^{j} \mathcal A(\Theta)| \leqslant \overline D_{j}
\end{equation}
for some constants $\overline D_{j}>0$. We observe that, under this assumption, $(M,g)$ must have bounded geometry by the Gauss equation (and its iterated derivatives) and Thm.~2.1 of \cite{chen2007uniqueness}.

\begin{theorem}\label{App-noncompact-case}
Assume that $(M,g)$ is an $n(\geqslant 3)-$dimensional complete noncompact Riemannian manifold, and let $(\overline{g}(t),\overline{w}(t))$ be a shrinking self-similar solution to the extended Ricci flow satisfying \eqref{Additional-BG} and \eqref{CAdd-Thm3} on $M$ with potential function $\overline{f}$ and initial value $(g,w)$. Given an $(n-1)-$dimensional compact smooth manifold $\Sigma$ without boundary, and let $\mathscr F$ be the MCF of $\Sigma$ in a gradient shrinking extended Ricci soliton background which develops a singularity of type-I. Consider the  normalized MCF $\widetilde{\mathscr F}$ in $(M,g)$. In addition, assume that there exists a point $q_{0}\in \Sigma$ such that the reduced distance $\ell_{x_{t}(q_{0}),t}$ converges pointwise to $f$ (as $t\to T$) on $M\times[0,T).$ Then, for any sequence $s_{1}<s_{2}<\cdots<s_{j}<\cdots \rightarrow \infty$ there exists a subsequence $s_{j_{k}}$ such that the family of immersion maps $\widetilde x_{s_{j_k}}: \Sigma\to (M,g)$ from pointed manifold $\big(\Sigma ,q_0\big)$ converges to an immersion map $x_\infty: \Sigma_\infty \to (M,g)$ from an $(n-1)-$dimensional complete pointed Riemannian manifold $\left(\Sigma_{\infty}, x_\infty ^* g, q_{\infty}\right)$ in the $C^\infty$ Cheeger–Gromov sense. Furthermore, $(\Sigma_{\infty}, x_\infty^* g)$ is an $f_\infty-$minimal hypersurface of  $(M, g),$ where $f_\infty = f \circ x_\infty$. 
\end{theorem}
\begin{proof}
As we had already mentioned in the proof of the main theorem, it is enough to show that $\{\widetilde x_{s_{j}}(q_{0})\}_{j = 1}^\infty$ is a bounded sequence in $(M,g)$, since the remainder of the proof is the same as in the compact case.

We start by taking $t_1$, $t_2$ with $0\leqslant t_{1}<t_{2}<T$, and $\{x_t(q_0)\}_{t\in[t_1, t_2]}$ as a curve joining $x_{t_1}(q_0)$ and $x_{t_2}(q_0)$. Thus, 
\begin{eqnarray*}
\ell_{x_{t_2}(q_0),t_2}(x_{t_1}(q_0),t_1)&\leqslant& \frac{1}{2\sqrt{t_2-t_1}}\int_{t_{1}}^{t_{2}}\sqrt{t_{2}-t}\left( S_{\overline g} + \left|\frac{\partial x_{t}}{\partial t}\right|^2 \right)dt\\
&=& \frac{1}{2\sqrt{t_{2}-t_{1}}}\int_{t_{1}}^{t_{2}}\sqrt{t_{2}-t}\bigl( S_{\overline g} + H_{\overline g}^2 \bigr)dt.
\end{eqnarray*}
Since $\mathscr F$ develops a singularity of type-I, we have $(T-t)H_{\overline g}^2$ is bounded. Moreover, the bounded geometry assumption and the fact that $(T-t)S_{\overline g}=S_g$ imply $S_{\overline g} + H_{\overline g}^2\leqslant \frac{C}{T-t}$ for some positive constant $C$ which does not depend on $t.$ Hence,
\begin{eqnarray*}
&&\ell_{x_{t_2}(q_{0}),t_{2}}(x_{t_{1}}(q_0),t_{1})\\
&\leqslant&  \frac{C}{2\sqrt{t_{2}-t_{1}}}\int_{t_{1}}^{t_{2}}\frac{\sqrt{t_{2}-t}}{T-t}dt
\leqslant   \frac{C}{2\sqrt{t_{2}-t_{1}}}\int_{t_{1}}^{t_{2}}\frac{1}{\sqrt{T-t}}dt
\leqslant  C\frac{\sqrt{T-t_{1}}}{\sqrt{t_{2}-t_{1}}}. 
\end{eqnarray*}
By assumption $\ell_{x_t(q_0),t}$ converges pointwise to $f$ (as $t\to T$) on $M\times[0,T)$ and by taking the limit as $t_{2}\to T$, we have 
$f(x_{t_1}(q_0),t_1)\leqslant C$. As $f(x_t(q_0),t)=\overline f(x_t(q_0))=f(\widetilde x_s(q_0))$, one has $f(\widetilde x_s(q_0))\leqslant C$
for all $s\in[-\log T,\infty)$. Thm.~5.1 in  Wang~\cite{wang2016ricci} ensures that there exist positive constants $C_1$ and $C_2$ such that 
\begin{equation*}
\frac{1}{4}(r-C_1)^2\leqslant f\leqslant \frac{1}{4}(r+C_2)^2
\end{equation*}
on $M$, where $r(q)=d_g(q_0, q)$ is the distance function from any fixed point $q_0\in M$. Then
\begin{equation*}
d_g(q_0, \widetilde x_s(q_0))\leqslant 2\sqrt{C}+C_{1}, 
\end{equation*}
which means that $\{\widetilde x_{s_{j}}(q_{0})\}_{j = 1}^\infty$ is bounded in $(M,g)$ and the proof is complete.
\end{proof}

\section{Concluding remarks} \label{Concluding remarks}
We start this section by appending an example of $f-$minimal hypersurface of an Euclidean spherical cap. Next, we show how to construct a family of mean curvature solitons for the MCF in a gradient extended Ricci soliton background by means of radial smooth functions on Euclidean space. For explicit examples of MCF in a gradient extended Ricci soliton background by means of invariance under the action of a translation group in Euclidean space, see~\cite{gomes2023mean}. 

\begin{example}\label{example:f-minimal}
Consider $\mathbb S^n_\epsilon = \{x= (x_1, \ldots, x_{n +1}) \in \mathbb S^n ; x_{n + 1} \geqslant \epsilon \}$ the Euclidean spherical cap with boundary $\partial \mathbb S^n_\epsilon$, $\ 0<\epsilon<1$. Let us consider the function $f(x) = \frac{\epsilon(n - 1)}{\epsilon^2 - 1} h_v(x)$ on $\mathbb S_\epsilon^n$, where $h_v(x)=\langle x, v \rangle$ is the height function on $\mathbb S^n\subset \mathbb{R}^{n+1}$ with respect to $v=e_{n + 1}\in\mathbb R^{n+1}.$ 
So, $\partial \mathbb S^n_\epsilon = h_v^{-1}\{\epsilon\},$ and then  $|\nabla h_v|$ is a  constant $c$ when restrict to $\partial\mathbb S^n_\epsilon$, and $e_0 = \frac{\nabla h_v}{c}$ is the inward unit vector field along $\partial \mathbb S^n_\epsilon.$ Moreover, it is known that
$\operatorname{Hess}_{g_0} h_v = - h_vg_0$ on $\mathbb S^n_\epsilon$, where $g_0$ is the round metric, and the second fundamental form of $\partial\mathbb S^n_\epsilon$ is given by 
\begin{align}\label{SecondForm-Scap}
\mathcal A = -\nabla e_0 = -\dfrac{1}{c}\operatorname{Hess}_{g_0} h_v = \dfrac{h_v}{c}g_0 = \dfrac{\epsilon}{c}g_0   
\end{align}
that implies $H = \dfrac{\epsilon(n - 1)}{c}.$ Now, note that we can write $\overline\nabla f = \nabla f + \left\langle \overline\nabla f, \Vec{x} \right\rangle \Vec{x}.$ Since 
\[
f(x) = \frac{\epsilon(n - 1)}{ \epsilon^2 - 1}\langle x, e_{n + 1}\rangle=  \frac{\epsilon(n - 1)}{\epsilon^2 - 1} x_{n + 1},
\]
we have
\begin{align*}
\overline\nabla f = \frac{\epsilon(n - 1)}{\epsilon^2 - 1} e_{n + 1}\quad\hbox{and}\quad \nabla f = \frac{\epsilon(n - 1)}{ \epsilon^2 - 1}\left(-x_{n + 1}x_1, \ldots,-x_{n + 1}x_n, 1 -x_{n + 1}^2 \right).
\end{align*}
Hence, along $\partial\mathbb S^n_\epsilon$
$$e_0 (f) = \dfrac{1}{c} \nabla f\langle \Vec{x}, e_{n +1}\rangle =\dfrac{1}{c}  \left\langle \nabla f, e_{n + 1}\right\rangle = \dfrac{\epsilon(n - 1)}{c(\epsilon^2 - 1)}(1 -\epsilon^2) = -\dfrac{\epsilon(n - 1)}{c}.$$
Thus, the boundary $\partial\mathbb S^n_\epsilon$ is a $f-$minimal hypersurface of $\mathbb S^n_\epsilon.$ 
\end{example}

\begin{remark}
We observe that Example~\ref{example:f-minimal} corrects a mistake in Example~4 of \cite{gomes2023mean}.
\end{remark}

In what follows, we are using the approach of Section~7 in \cite{gomes2023mean} to show how to obtain explicit parameter functions for constructing mean curvature solitons for the MCF in a gradient extended Ricci soliton background by means of radial smooth functions on Euclidean space.

Let $g_{ij} = \frac{1}{F^2}\delta_{ij}$ be a Riemannian metric on $\mathbb R^n$, where $F$ is a nonzero smooth function on $\mathbb R^n$. For this metric, we want to obtain parameter functions for
\begin{subequations}\label{Bizu}
\begin{empheq}[left=\empheqlbrace]{align}
&\operatorname{Ric}_{g} + \operatorname{Hess}_{g} f - \alpha_n d w  \otimes  d w = \lambda g,\label{Bizu11}\\
&\Delta_{g} w = \langle\nabla_{g} f, \nabla_{g} w\rangle_{g}\label{Bizu12}.
\end{empheq}
\end{subequations}
Since the metric $g_{ij}$ is conformal to $\delta_{ij}$, it is well known (see, e.g., \cite{besse2007einstein}) that 
\begin{align*}
(\operatorname{Ric}_{g})_{ij} &=\frac{1}{F^2}\Big[(n - 2)FF_{x_ix_j} + \Big(F\sum_k F_{x_k x_k} - (n - 1) \sum_k F^2_{x_k}\Big)\delta_{ij}\Big]\quad \forall i,j\\
(\operatorname{Hess}_g h)_{ij} &= h_{x_ix_j} +\frac{F_{x_j}}{F} h_{x_i} +\frac{F_{x_i}}{F} h_{x_j} \quad \forall i \neq j\\
(\operatorname{Hess}_g h)_{ii} &= h_{x_ix_i} + 2\frac{F_{x_i}}{F} h_{x_i} -\sum_k\frac{F_{x_k}}{F} h_{x_k} \quad \forall i
\end{align*}
for any smooth function $h$ on $\mathbb R^n$. Hence,
\[\Delta_{g} h = F^2 \Big(\sum_k h_{x_kx_k} + (2 - n)  \sum_k\frac{F_{x_k}}{F}h_{x_k}\Big).\]

\begin{proposition}\label{Prop-radFunc}
Consider $\mathbb R^n$ with the metric $g_{ij} = \frac{1}{F^2(r)}\delta_{ij}$, for some nonzero smooth function $F$ on $\mathbb R^n$ which depends only on $r  = \|x\|$. We can obtain smooth functions  $f (r)$ and $ w(r)$ satisfying~\eqref{Bizu11} (as well~\eqref{Bizu12}) by means of the system
\begin{align*}
\left\{
\begin{array}{rcl}
\dfrac{(2n - 3)F'} {rF}  + \dfrac{f'}{r} +\dfrac{F''}{F} - (n - 1)\Big(\dfrac{F'}{F}\Big)^2 -\dfrac{F'}{F}f'  &=& \dfrac{\lambda}{F^2}\\[2ex]
w'' + \Big(\dfrac{n - 1}{r}  - (n-2)\dfrac{F'}{F} - f'\Big)w'  &=& 0
\end{array}
\right.
\end{align*}
in $\mathbb{R}^n\setminus\{0\}$, where the superscript $'$ denotes the derivative with respect to $r.$
\end{proposition}

\begin{proof}
We need to analyze \eqref{Bizu11} in two cases. For $i \neq j$, it rewrites as
\begin{align}\label{Holds}
(n - 2) \frac{F_{x_ix_j}}{F} + f_{x_i x_j} +\frac{F_{x_j}}{F} f_{x_i} +\frac{F_{x_i}}{F}f_{x_j} - \alpha_n w_{x_i} w_{x_j} = 0,
\end{align}
and for $i=j$,
\begin{align}\label{BIZU1}
\nonumber &(n - 2) \frac{F_{x_ix_i}}{F} + \sum_k \frac{ F_{x_k x_k} }{F} - (n - 1)\sum_k\frac{F^2_{x_k}}{F^2} + f_{x_i x_i} +2\frac{F_{x_i}}{F} f_{x_i}\\
&-\sum_k\frac{F_{x_k}}{F}f_{x_k} - \alpha_n w^2_{x_i}  = \frac{\lambda}{F^2}.
\end{align}
While \eqref{Bizu12}, it rewrites as
\begin{align}\label{add:Naza}
\sum_k w_{x_kx_k} + (2 - n)  \sum_k\frac{F_{x_k}}{F}w_{x_k} = \sum_kf_{x_k}w_{x_k}.
\end{align}
Next observe that for any radial smooth function $h(r)$ on $\mathbb{R}^n$, we have $h_{x_i}= h'x_i/r$ and $h_{x_ix_j} =x_ix_j ( h''/r^2 - h'/r^3)$, for all $i \neq j$. Besides, $h_{x_i}= h'x_i/r$ and $h_{x_ix_i} =  x_i^2(h''/r^2 - h'/r^3) +  h'/r ,$  for all $i.$ Thus, from \eqref{Holds} and~\eqref{BIZU1}, we obtain
\begin{align}\label{GuerraeSUA}
x_ix_j\Big[\frac{n - 2 }{F}\Big(\frac{F''}{r^2} - \frac{F'}{r^3}\Big) + \frac{f''}{r^2} - \frac{f'}{r^3} +2\frac{F'}{F} \frac{f'}{r^2}- \dfrac{\alpha_n}{r^2} {w'}^2\Big] = 0 \quad\forall i\neq j
\end{align}
and
\begin{align}\label{GuerraeSUA1}
&x_i^2\Big[\frac{n - 2} {F}\Big(\frac{F''}{r^2} - \frac{F'}{r^3}\Big)  + \frac{f''}{r^2} - \frac{f'}{r^3} +2\frac{F'}{F} \frac{f'}{r^2}  - \frac{\alpha_n}{r^2} {w'}^2  \Big] +  \frac{(n - 2)F'} {rF}  \nonumber\\
&+ \frac{f'}{r} +\frac{F''}{F} 
+ \frac{(n - 1)F'}{rF}    - (n - 1)\Big(\frac{F'}{F}\Big)^2 -\frac{F'}{F}f'  = \frac{\lambda}{F^2}\quad \forall i.
\end{align}
From \eqref{GuerraeSUA} we prove that~\eqref{GuerraeSUA1} reduces to
 \begin{align*}
&\frac{(n - 2)F'} {rF}  + \frac{f'}{r} +\frac{F''}{F} 
+ \frac{(n - 1)F'}{rF}    - (n - 1)\Big(\frac{F'}{F}\Big)^2 -\frac{F'}{F}f'  = \frac{\lambda}{F^2}.
\end{align*}
Likewise, \eqref{add:Naza} becomes
\begin{align*}
w'' + \Big(\dfrac{n - 1}{r}  - (n-2)\dfrac{F'}{F}\Big)w'  = f'w'.
\end{align*}
This proves the proposition.
\end{proof}

For constructing a family of mean curvature solitons for the MCF in the corresponding self-similar solution to the $(\overline g, \overline w)$-extended Ricci flow background on $M$, it is enough to use Proposition~\ref{Prop-radFunc} and consider an $f$-minimal hypersurface $\Sigma$ of $M$ (although we know that hard work is needed to find $f$-minimal hypersurfaces for this case), and then one can proceed as in \cite[Thm.~3]{gomes2023mean} to obtain such a family.

It would be great if Example~\ref{example:f-minimal} could be used to construct a family of mean curvature solitons for the MCF in an extended Ricci flow background. However, if we could use the parameter functions given in this example in Proposition~\ref{Prop-radFunc}, then we would also allow $\lambda$ to be a nonconstant radial smooth function, which may be useful in further studies. Indeed, we know that the stereographic projection 
$$\pi^{-1}:\big(\mathbb{R}^n, g\big)\to \mathbb{S}^n\setminus \{(0,\ldots,-1)\}$$
given by $$\pi^{-1}(x_1,\ldots,x_n)=\Big(\frac{2x_1}{1+r^2},\ldots,\frac{2x_n}{1+r^2}, \frac{1-r^2}{1+r^2}\Big)$$
is an isometry, where $g$ is given by $\frac{4}{(1+r^2)^2}\delta_{ij}$. Then, we can take $\Sigma =\pi(\partial\mathbb S^n_\epsilon),$  $\widetilde{f}=f\circ\pi^{-1} = \frac{\epsilon(n - 1)(1 - r^2)}{(\epsilon^2 - 1)(1 + r^2)}$ and $F(r)=\frac{1+r^2}{2}$ so that $\Sigma$ is a $\widetilde{f}$-minimal hypersurface of $(\mathbb{R}^n, g)$. Now, we define $\overline g(t) := \kappa (T - t) \psi^*_t g$ and  $x(\cdot, t) := \psi(\cdot, -t + 2(T -\kappa)),$  for $\kappa = 1,$ with $t \in \big(2(T-1), T\big)$ and $\psi_{T - 1} = \operatorname{Id}$ in the shrinking case; and for $\kappa = -1,$ with $t \in \left(T, 2(T+1)\right)$ and $\psi_{T + 1} = \operatorname{Id}$ in the expanding case, where $\psi_t$ is a smooth one-parameter family of diffeomorphisms of $\mathbb R^n$ generated from the flow of $\nabla_g \widetilde{f}/\kappa (T - t).$ By taking $F$ and $\widetilde f$ in Proposition~\ref{Prop-radFunc} we get  $w''+\psi(r)w'=0$, for some $\psi(r)$, from which we obtain $w(r)$ and $\lambda(r)=n-1 +\frac{\varepsilon}{\varepsilon^2-1}(n-1)(r^2-1).$ Besides, from \eqref{SecondForm-Scap}, we have $|\mathcal A(\cdot, t)| = \frac{\epsilon (n - 1)}{c\sqrt{T - t}}$ on $\Sigma_t$, and then it would be the corresponding type-I singularity for the shrinking case.

\section{Acknowledgments}
José N.V. Gomes has been partially supported by Conselho Nacional de Desenvolvimento Científico e Tecnológico (CNPq), Grant 310458/2021-8, and Fundação de Amparo à Pesquisa do Estado de São Paulo (FAPESP), Grants 2023/11126-7, 2022/16097-2 and 2024/00923-6. Matheus Hudson has been partially supported by Fundação de Amparo à Pesquisa do Estado de São Paulo (FAPESP), Grants 2023/13921-9 and  2024/20015-7. Hikaru Yamamoto has been partially supported by JSPS KAKENHI Grant-in-Aid for Early-Career Scientists 22K13909. The authors would like to thank the referee for the careful reading and for the useful comments that improved the paper.

\end{document}